\documentclass[12pt]{article}
\usepackage[utf8]{inputenc}
\usepackage[T1]{fontenc}

\usepackage{geometry}
\usepackage{amsmath,amsfonts,amssymb,amsthm}
\usepackage{color}
\usepackage{hyperref}
\usepackage{longtable}
\usepackage{enumerate}
\usepackage[all]{xy}

\usepackage[pdftex]{graphicx}

\def\R{\mathbb R}

\def\es{\mathbb S}
\def\s{\mathcal S}

\def\a{\mathcal A}
\def\b{\mathcal B}
\def\g{\mathcal G}

\def\r{\mathcal R}

\def\h{\mathcal H}

\def\w{\mathcal W}
\def\v{\mathcal V}

\newtheoremstyle{plain}{10pt}{10pt}{\it}{0pt}{\bf}{.}{0.5em}{}
\newtheorem{twr}{Theorem}

\newtheorem{lemat}[twr]{Lemma}

\newtheorem{klaim}[twr]{Claim}

\newtheoremstyle{definition}{10pt}{10pt}{}{0pt}{\bf}{.}{0.5em}{}
\theoremstyle{definition}
\newtheorem{problem}[twr]{Problem}
\newtheorem{df}[twr]{Definition}

\newtheorem{uwaga}[twr]{Remark}

\author{\Large Justyna~Ogorzały\footnote{Institute of Mathematics, Jagiellonian University in Krak\'ow, ul. prof. S. Lojasiewicza 6, 30-348, 
Kraków, Poland, {\bf justyna.ogorzaly@gmail.com}}} 
\title{\bf \textsc{Quasistatic contact problem with unilateral constraint for elastic-viscoplastic materials}}
\begin{document}
\maketitle

{\bf Abstract.} This paper consists of two parts. In the first part we prove the unique solvability for the abstract variational-hemivariational inequality with history-dependent operator. The proof is based on the existing result for the static variational-hemivariational inequality and a fixed point argument. In the second part, we consider a mathematical model which describes quasistatic frictional contact between a deformable body and a rigid foundation. In the model the material behaviour is modelled by an elastic-viscoplastic constitutive law. The contact is described with a normal damped response, unilateral constraint and memory term. In the analysis of this model we use the abstract result from the first part of the paper.
\section{Introduction}
Many mechanical problems involving nonmonotone, multivalued relations between stresses and strains, between reactions and displacements or between
generalized forces and fluxes. These relations expressed in terms of nonconvex superpotentials (cf. \cite{Pan1, Pan2}) lead to hemivariational inequalities. Let us add, that the nonconvex superpotentials (cf. \cite{clarke}) generalize
the notion of convex superpotential introduced by Moreau \cite{MOR}. The convex superpotentials describe monotone possibly
multivalued mechanical laws and they lead to variational inequalities. The variational-hemivariational inequalities were introduced by Panagiotopoulos and they represent a special class of inequalities, in which both convex and nonconvex functions occur.
These type of inequalities  are a useful tool in the study of nonsmooth variational problems with constraints and boundary value problems with discontinuous nonlinearities. The results associated with variational-hemivariational inequalities and its applications can be found in the monographs, e.g. \cite{CLM, GMDR, GM, NP, Pan}.

The aim of this paper is to study the existence and uniqueness of the solution of the variational-hemivariational inequality with history-dependent operator and to apply obtained result into the analysis of a quasistatic contact problem for elastic-viscoplastic materials. It should be noted that the existence and uniqueness result for the static variational-hemivariational inequality without history-dependent operator is obtained by Mig\'orski et al. in \cite{MOSVHV}. This paper generalizes the result from \cite{MOSVHV}. The first novelty in our work is that we consider the variational-hemivariational
inequality defined on a bounded interval of time. The second novelty related to the special structure of the variational-hemivariational inequality which we consider. Namely, our inequality contains convex and nonconvex functionals and, moreover, it contains so-called history-dependent operator which at any moment $t \in (0,T),$ depend on the history of the solution up to the moment $t$. Furthermore, we present the example of a contact problem which leads to the variational-hemivariational inequality with history-dependent operator. 

The rest of the paper is structured as follows. Section 2 contain notation and definitions. In
Section 3 we consider the abstract problem and we prove it unique solvability. Finally, in Section 4 we apply the result obtained in Sections 3 in the analysis of the contact problem.

\section{Preliminary}
We introduce the notation and we recall some preliminary material which
will be used in the next parts of this paper.

Let $V$ and $X$ are separable and reflexive Banach spaces with the duals $V^*$ and $X^*$, respectively, and $K \subset V$. We consider also the space $\v = L^2(0,T;V),$  where $0 < T < + \infty.$ Moreover, by $\mathcal{L}(V,X)$ we denote a space of linear and bounded operators with a Banach space $V$ with values in a Banach space $X$ with the norm $\Vert \cdot \Vert_{\mathcal{L}(V,X)}.$ The duality pairing between $X^*$ and $X$ is denoted by $\langle \cdot, \cdot \rangle_{X^* \times X},$ whereas the duality pairing between $\v^*$ and $\v$ is given by
$
\langle u,v \rangle_{\v^* \times \v} = \int\limits_{0}^{T}\,\langle u(t), v(t) \rangle_{V^* \times V}\,dt$ for $u \in \v^*, v \in \v.$ If $X$ is a Hilbert space thus the inner product is denoted by $(\cdot,\cdot)_X.$

We use the following concepts of the generalized directional derivative, the Clarke sub\-dif\-feren\-tial and the subgradient of a convex function.
\begin{df}\label{GENERAL}
The generalized directional
derivative (in the sense of Clarke) of a locally Lipschitz function $\varphi \colon X \longrightarrow \mathbb{R}$ at the point $x \in X$ in the direction $v \in X,$ denoted by $\varphi^{0} (x;v)$ is defined by
$$
\varphi^{0} (x;v) = \limsup_{y \to x, \, \lambda \downarrow 0} \frac{\varphi(y + \lambda v) - \varphi (y)}{\lambda}.
$$
\end{df}
\begin{df}\label{DDEF2}
Let $\varphi \colon  X \longrightarrow \mathbb{R}$ be a locally Lipschitz function. The Clarke generalized gradient (subdifferential) of $\varphi$ at $x \in X,$ denoted by $\partial \varphi (x),$ is the subset of a dual space $X^{*}$ defined by
$
\partial \varphi (x) = \lbrace \zeta \in X^{*} \, \vert \, \varphi^{0} (x;v) \geqslant \langle \zeta, v \rangle_{X^{*} \times X} \ \mbox{for all} \ v \in X \rbrace .
$
\end{df}
\begin{df}\label{DDEF1}
Let $\varphi \colon X \to \R \cup \{ +\infty \}$ be a proper,
convex and lower semicontinuous function. 
The subdifferential $\partial \varphi$ is generally a multivalued 
mapping $\partial \varphi \colon X \to 2^{X^*}$ defined by
$
\partial \varphi (x) = \{ \, x^*\in X^* \mid \langle x^*, v -x \rangle_{X^* \times X} 
\leqslant \varphi(v)-\varphi(x) \ \mbox{for all} \ v \in X \, \}
$
for $x \in X$.
The elements of the set $\partial \varphi (x)$ 
are called subgradients of $\varphi$ in $x$.
\end{df}
In this paper by $c$ we will denote a positive constant which can change from line to line. The following lemma is a consequence of the Banach contraction principle.
\begin{lemat}\label{lematkulig}
Let $X$ be a Banach space with a norm $\Vert \cdot \Vert_X$ and $T>0$. Let $\Lambda : L^2 (0,T;X) \longrightarrow L^2 (0,T;X)$ be an operator satisfying
$
\Vert (\Lambda \eta_1)(t) - (\Lambda \eta_2)(t) \Vert^2_X \leqslant c \int\limits_{0}^{t} \Vert \eta_1(s) - \eta_2 (s) \Vert^2_X\,ds
$
for every $\eta_1, \eta_2 \in L^2 (0,T;X),$ a.e. $t \in (0,T).$ Then $\Lambda$ has a unique fixed point in $L^2(0,T;X),$ i.e., there exists a unique $\eta^* \in L^2(0,T;X)$ such that $\Lambda \eta^* = \eta^*.$
\end{lemat}
Now, we recall the concept of the history-dependent operator.
\begin{df}
An operator $\s \colon \v \longrightarrow \v^*$ that satisfies the inequality 
\begin{equation}\label{wS}
\Vert (\s u_1)(t) - (\s u_2) (t) \Vert_{V^*} \leqslant L_{\s} \int\limits_{0}^{t} \Vert u_1 (s) - u_2 (s) \Vert_{V}\, ds
\end{equation}
$
{\rm for}\ u_1, u_2 \in \v,\  {\rm for\ a.e.}\ t \in (0,T)\ {\rm with}\ L_{\s} >0,
$
is called {\textit{the history-dependent operator}}.
\end{df}
The following property of the history-dependent operators will be used later.
\begin{lemat}\label{PwS}
Let $\s_1,\s_2 \colon \v \longrightarrow \v^*$ be the operators which satisfy \eqref{wS}, then the operator $\s \colon \v \longrightarrow \v^*$ given by $(\overline{\s}u)(t) = (\s_1 u)(t) + (\s_2 u)(t)$ for $u \in \v,$ satisfies \eqref{wS}.
\end{lemat}
\begin{proof}
The proof is straightforward so we omit it.
\end{proof}

Finally, we present the result which concerns the existence and uniqueness of the solution of the static variational-hemivariational inequality.  Consider the following abstract problem.
\begin{problem}\label{probabst}
Find an element $u \in V$ such that $u \in K$ and
$$
\langle Au, v - u \rangle_{V^* \times V} + \varphi(u,v) - \varphi(u,u) + J^0(Mu;Mv - Mu) \geqslant \langle f, v -u \rangle_{V^* \times V} \ {\rm for\ all}\ v \in K.
$$
\end{problem}
We introduce the following hypotheses.
\begin{equation}\label{A0}
\left.
\begin{array}{l}
A: V \longrightarrow V^*\ \mbox{is such that}\hspace{0.5cm}\\
\ \ {\rm (a)}\ 
A\ \mbox{is pseudomonotone}.\hspace{0.5cm}\\
\ \ {\rm (b)}\ 
A\ \mbox{is coercive, i.e., there exist}\ \alpha_A > 0, \beta ,\beta_1 \in \R\ \mbox{and}\  u_0 \in K\ \mbox{such that}\hspace{0.5cm}\\
\hspace{1cm}\langle Av, v-u_0 \rangle_{V^* \times V} \geqslant \alpha_A \Vert v \Vert^2_V - \beta \Vert v \Vert_V - \beta_1\ {\rm for\ all} \ v \in V.\hspace{0.5cm}\\
\ \ {\rm (c)}\ 
A\ \mbox{is strongly monotone, i.e., there exists}\ m_A > 0\ \mbox{such that}\hspace{0.5cm}\\
\hspace{1cm} \langle A v_1 - A v_2, v_1 - v_2 \rangle_{V^* \times V} \geqslant m_A \Vert v_1 - v_2 \Vert^2_V \
\mbox{for all}\ v_1, v_2 \in V.\hspace{0.5cm}
\end{array}
\right\}
\end{equation}

\begin{align}\label{fi0}
\left.
\begin{array}{l}
\varphi: K \times K \longrightarrow \R\ \mbox{is such that}\\
\ \ {\rm (a)}\ 
\varphi (u, \cdot): K \longrightarrow \R \ \mbox{is convex and lower semicontinuous on}\ K,\ \mbox{for all}\ u \in K.\\
\ \ {\rm (b)}\ 
\mbox{there exists}\ \alpha_{\varphi} > 0\ \mbox{such that}\\
\hspace{1cm} \varphi (u_1, v_2) - \varphi (u_1, v_1) + \varphi (u_2, v_1) - \varphi (u_2, v_2) \leqslant \alpha_{\varphi} \Vert u_1 - u_2 \Vert_V \Vert v_1-v_2 \Vert_V\\
\hspace{1cm} \mbox{for all}\ u_1, u_2, v_1, v_2 \in K.
\end{array}
\right\}
\end{align}

$\\$

\begin{align}\label{j0}
\left.
\begin{array}{l}
J\colon X \longrightarrow \R\ \mbox{is such that}\\
\ \ {\rm (a)}\ 
J\ \mbox{is locally Lipschitz}.\hspace{1.4cm}\\
\ \ {\rm (b)}\ 
\Vert \partial J(v) \Vert_{X^*} \leqslant c_0 +  c_1\,\Vert v \Vert_X \ {\rm for\ all}\ v \in X\ \mbox{with}\ c_0, c_1 \geqslant 0.\hspace{1.4cm}\\
\ \ {\rm(c)}\ 
\mbox{there exists}\ \alpha_J > 0\  \mbox{such that}\\
\hspace{1cm} J^0(v_1;v_2-v_1) + J^0(v_2;v_1-v_2) \leqslant \alpha_J \Vert v_1-v_2 \Vert_X^2\ \mbox{for all}\ v_1, v_2 \in X.\hspace{1.4cm}
\end{array}
\right\}
\end{align}
\begin{equation}\label{opm}
M \colon V \longrightarrow X \ \mbox{is a linear, continuous and compact operator.}\qquad\qquad\qquad\qquad\qquad
\end{equation} 
\begin{equation}\label{k0}
K \ \mbox{is a nonempty, closed and convex subset of}\ V. \qquad \qquad \qquad \qquad \qquad \qquad \qquad 
\end{equation}
\begin{equation}\label{f0}
f \in V^*.
\end{equation}
\begin{uwaga}\label{RELAXED}
Hypothesis (\ref{j0})(c) is used in the proof of the uniqueness of solution to hemivariational inequalities.  
This hypothesis is equivalent to the following condition 
\begin{align}\label{REQUIV}
\langle z_1 - z_2, v_1-v_2 \rangle_{X^* \times X} \geqslant - \alpha_J \,\Vert v_1 - v_2 \Vert_X^2
\end{align}
for all $z_i \in \partial J (v_i), z_i,v_i \in X, i=1,2$ with $\alpha_J > 0$. This condition 
is called {\it the relaxed monotonicity 
condition} for a locally Lipschitz function $J$. It can be proved that for a convex function, condition (\ref{j0})(c), or equivalently  \eqref{REQUIV}, holds with $\alpha_J = 0.$
\end{uwaga}
\begin{twr}\label{twabst}
Under hypotheses \eqref{A0}--\eqref{f0} and
\begin{equation}\label{w0}
m_A > \alpha_{\varphi} + \alpha_J\,\Vert M \Vert^2,\ \ \alpha_A > 2\,\alpha_J \,\Vert M \Vert^2
\end{equation}
Problem \ref{probabst} has a unique solution $u \in V.$
\end{twr}
\begin{proof}
The proof of Theorem \ref{twabst} is similar to the proof of Theorem 16 in \cite{MOSVHV}. 
\end{proof}
\section{History-dependent variational-hemivariational inequality}\label{s31}
In this section, we study an abstract variational-hemivariational inequality which contains a history-dependent operator. We start with the time-dependent version of Problem \ref{probabst}. To this end, we consider the operators $A\colon (0,T)\times V \longrightarrow V^*,\ M\colon V \longrightarrow X,$ the functional $J\colon (0,T)\times X \longrightarrow \R,$ the functions $\varphi\colon K \times K \longrightarrow \R$ and $f\colon (0,T) \longrightarrow V^*.$ 
With these data we deal with the following variational-hemivariational inequality in which the time variable plays the role of parameter.
\begin{problem}\label{probabst1}
Find $u \in \v$ such that $u(t) \in K$ and
\begin{align}\label{vhv}
\begin{split}
 \langle A(t, u(t)), v -u(t)\rangle_{V^* \times V} &+ \varphi (u(t),v) - \varphi (u(t), u(t))\\ &+ J^0 (t, Mu(t);M(v-u(t))) \geqslant \langle f(t), v-u(t) \rangle_{V^* \times V}
\end{split}
\end{align}
for all $v \in K$ and a.e. $t \in (0,T).$
\end{problem}
In the study of Problem \ref{probabst1}, we assume that the assumptions  \eqref{fi0}, \eqref{opm} and \eqref{k0} hold. Moreover, we need the following assumptions on the data.
\begin{align}\label{wA}
\left.
\begin{array}{l}
A\colon (0,T) \times V \longrightarrow V^*\ \mbox{is such that}\\
\ \ {\rm (a)}\ 
A(\cdot,v)\ \mbox{is measurable on}\ (0,T)\ \mbox{for all}\ v \in V.\\
\ \ {\rm (b)}\ 
A(t,\cdot)\ \mbox{is strongly monotone, i.e., there exists}\ m_A > 0\ \mbox{such that}\\
\hspace{1cm} \langle A (t, v_1) - A (t, v_2), v_1 - v_2 \rangle_{V^* \times V} \geqslant m_A \Vert v_1 - v_2 \Vert^2_V \\
\hspace{1cm} \mbox{for all}\ v_1, v_2 \in V\ \mbox{and a.e.}\ t \in (0, T). \\
\ \ {\rm (c)}\ 
A(t,\cdot)\ \mbox{is continuous on}\ V\ \mbox{for a.e.}\ t \in (0,T).\\
\ \ {\rm (d)}\ 
\Vert A(t, v) \Vert_{V^*} \leqslant a_0 (t) + a_1 \Vert v \Vert_V\ \mbox{for all}\ v \in V,
\ \mbox{a.e.}\ t \in (0,T)\\ 
\hspace{1cm} \mbox{with}\ a_0 \in L^2 (0, T), a_0 \geqslant 0\ \mbox{and}\ a_1 >0.\\
\ \ {\rm (e)}\ 
A(t,\cdot)\ \mbox{is coercive, i.e., there exists}\ \alpha_A > 0,\ \beta \in \R,\  \beta_1(t) \in L^2(0,T)\\  \hspace{1cm} \mbox{and}\ u_0 \in K\ \mbox{such that}\ 
\langle A(t,v),v - u_0 \rangle_{V^* \times V} \geqslant \alpha_A\,\Vert v \Vert^2_V - \beta \Vert v \Vert_V - \beta_1(t)\\ \hspace{1cm} {\rm for\ all}\ v \in V \ {\rm a.e.} \ t \in (0,T).
\end{array}
\right\}
\end{align}
\begin{align}\label{wJ}
\left.
\begin{array}{l}
J\colon  (0,T) \times X \rightarrow \mathbb{R}\ \mbox{is such that}\hspace{1.5cm}\\
\ \ {\rm (a)}\ 
J(\cdot,v)\ \mbox{is measurable on}\ (0,T)\ \mbox{for all}\ v \in X.\hspace{1.5cm}\\
\ \ {\rm (b)}\ 
J (t, \cdot)\ \mbox{is locally Lipschitz on}\ X\ \mbox{for a.e.}\ t \in (0,T).\hspace{1.5cm}\\
\ \ {\rm (c)}\ 
\Vert \partial J (t,v) \Vert_{X^*} \leqslant c_0(t) + c_1 \Vert v \Vert_X\ \mbox{for all}\ v \in X,\ \mbox{a.e.}\ t \in (0,T)\ \mbox{with}\\ \hspace{1cm} c_0 \in L^2(0,T),\; c_0,c_1 \geqslant 0.\hspace{1.5cm}\\
\ \ {\rm (d)}\  
J(t, \cdot)\ \mbox{or}\ - J (t, \cdot)\ \mbox{is regular (in the sense of Clarke) on}\ X\ \mbox{ for}\hspace{1.5cm}\\ 
\hspace{1cm} \mbox{a.e.}\ t \in (0,T).\hspace{1.5cm}\\
\ \ {\rm (e)}\ 
\mbox{there exists}\ m_J > 0\ \mbox{such that} \hspace{1.5cm}\\
\hspace{1cm} J^0 (t, v_1; v_2-v_1) + J^0(t,v_2;v_1-v_2)  \leqslant m_J \Vert v_1-v_2 \Vert^2_{X}\hspace{1.5cm}\\
\hspace{1cm} \mbox{for all}\ v_1, v_2 \in X\ \mbox{and a.e.}\ t \in (0,T).\hspace{1.5cm}
\end{array}
\right\}
\end{align}
Moreover, we assume that
\begin{align}\label{WF}
\left.
\begin{array}{l}
\ \ {\rm (a)}\ 
f \in \v^*. \hspace{0.4cm}\\
\ \ {\rm (b)}\ 
m_A >\alpha_{\varphi}+ m_J \Vert M \Vert^2,\ \ \alpha_A > 2\,m_J\,\Vert M \Vert^2,\ \ \mbox{where}\ \ \Vert M \Vert = \Vert M \Vert_{\mathcal{L} (V,X)}.\hspace{0.4cm}
\end{array}
\right\}
\end{align}
\par We have the following existence and uniqueness result.
\begin{twr}\label{twierdzenie1}
Under the assumptions \eqref{fi0}, \eqref{opm}, \eqref{k0} and \eqref{wA}--\eqref{WF}, 
Problem \ref{probabst1} has a unique solution $u \in \v$.
\end{twr}
\begin{proof}
We use Theorem \ref{twabst} for $t \in (0,T)$ fixed. Note that, from the hypothesis \eqref{wA}, it follows that the operator $A(t, \cdot)$ satisfies \eqref{A0} for a.e.\;$t \in (0,T).$ From \eqref{wA}(b),(c),(d) we observe, that
$A$ is monotone and  hemicontinuous and bounded. Hence and from Theorem~3.69 in \cite{MOSBOOK}, we know that the operator $A(t, \cdot)$ is pseudomonotone, so the condition \eqref{A0}(a) holds for a.e. $t \in (0,T)$.
Moreover, for a.e. $t \in (0,T),$ the condition \eqref{wA}(e) implies \eqref{A0}(b). 
We also see, that from the hypothesis \eqref{wJ} and \eqref{opm}, it follows that the function $ J(t, \cdot)$ satisfies \eqref{j0} for a.e. $t \in (0,T).$ Note that, the assumption \eqref{WF}(b) implies the assumption \eqref{w0} with $\alpha_J = m_J$. Hence, exploiting Theorem \ref{twabst}, we deduce that, for a.e. $t \in (0,T),$ Problem~\ref{probabst1} has a unique solution $u(t) \in K.$

Now, we prove that the function $t \longmapsto u(t)$  is measurable on $(0,T).$ Let $g \in V^*$ be given and $u(t)\in V$ be the unique solution of the inequality \eqref{vhv}.
We claim that the solution $u$ depends continuously on the right-hand side $g$, for a.e. $t \in (0,T).$ Namely, let $g_1, g_2 \in V^*$ and $u_1(t), u_2(t) \in K$ be the corresponding solutions to 
\eqref{vhv}. Then 
\begin{align}\label{ab}
\begin{split}
\langle A(t, &u_1(t)), v - u_1(t) \rangle_{V^* \times V} + \varphi(u_1(t),v) - \varphi(u_1(t),u_1(t))\\ &+ J^0 (t, M u_1(t);M(v-u_1(t))) \geqslant  \langle g_1, v - u_1(t) \rangle_{V^* \times V}
\end{split}
\end{align}
and
\begin{align}\label{cd}
\begin{split}
\langle A(t, &u_2(t)), v - u_2(t) \rangle_{V^* \times V} + \varphi(u_2(t),v) - \varphi(u_2(t),u_2(t))\\ &+ J^0 (t, M u_2(t);M(v-u_2(t))) \geqslant  \langle g_2, v - u_2(t) \rangle_{V^* \times V}
\end{split}
\end{align}
for all $v \in K$ and a.e. $t\in (0,T).$

We put $v = u_2(t)$ into \eqref{ab}  and $v = u_1(t)$ into \eqref{cd}. Adding the obtained inequalities, we get
\begin{align*}
\begin{split}
\langle A(t, &u_1(t)) - A(t, u_2(t)), u_1(t)-u_2(t) \rangle_{V^* \times V} - \big( \varphi(u_1(t),u_2(t)) - \varphi(u_1(t),u_1(t))\\ &+ \varphi(u_2(t),u_1(t)) - \varphi(u_2(t),u_2(t))\big) -\big( J^0(t,Mu_1(t);M(u_2(t)-u_1(t)))\\ &+ J^0(t,Mu_2(t);M(u_1(t)-u_2(t)))\big) \leqslant \langle g_1-g_2,u_1(t)-u_2(t) \rangle_{V^* \times V}.
\end{split}
\end{align*}
From this, conditions \eqref{fi0}(b), \eqref{wA}(b) and \eqref{wJ}(e), we have
\begin{align}
\begin{split}
m_A \Vert u_1(t) -u_2(t) \Vert_V^2 - \alpha_{\varphi}\Vert u_1(t) -u_2(t) \Vert_V^2 &-m_J\Vert M \Vert^2\Vert u_1(t) -u_2(t) \Vert_V^2\\ \leqslant &\Vert g_1 - g_2\Vert_{V^*}\Vert u_1(t) - u_2(t) \Vert_V.
\end{split}
\end{align}
Exploiting \eqref{WF}(b), we deduce that
\begin{equation}\label{oszacu}
\Vert u_1(t) - u_2(t) \Vert_V \leqslant c\,\Vert g_1 - g_2 \Vert_{V^*} \quad {\rm for\ a.e. }\ t \in (0,T).
\end{equation}
Hence, we conclude that the mapping $\psi \colon V^* \ni g \longmapsto u(t) \in V$ is continuous for a.e. $t \in (0,T),$ which proves the claim. By the condition \eqref{WF}(a) we know that the function $f\colon [0,T] \longrightarrow V^*$ is measurable. From Lemma 2.27(iii) in \cite{MOSBOOK}, we have that $\psi \circ f \colon [0,T] \longrightarrow V$ is measurable. So, the solution $u(t)$ of Problem \ref{probabst1} is measurable on $(0,T).$ 

Next, we prove that the solution of Problem \ref{probabst1} satisfies $u \in \v.$ Let $v_0 \in K.$ Thus, from  the inequality \eqref{vhv}, we get
\begin{align}\label{WZOR}
\begin{split}
 \langle A(t, &u(t)) - A(t,v_0), v_0 -u(t)\rangle_{V^* \times V} \leqslant \langle A(t,v_0), v_0 - u(t) \rangle_{V^* \times V} + \varphi (u(t),v_0)\\ &- \varphi (u(t), u(t)) + J^0 (t, Mu(t);M(v_0 -u(t))) + \langle f(t), v_0-u(t) \rangle_{V^* \times V}.
\end{split}
\end{align}
Now, we show the estimations which are needed in the next part of proof. Choosing $u_1 = u(t), u_2 = v_0, v_1 = u(t), v_2= v_0$ in \eqref{fi0}(b), we obtain
\begin{eqnarray*}
\varphi (u(t), v_0) - \varphi(u(t),u(t)) + \varphi(v_0, u(t)) - \varphi(v_0, v_0) \leqslant \alpha_{\varphi}\,\Vert u(t) - v_0 \Vert^2_V 
\end{eqnarray*}
and
\begin{eqnarray}\label{oszfi*}
\varphi (u(t), v_0) - \varphi(u(t),u(t)) \leqslant - \varphi(v_0, u(t)) + \varphi(v_0, v_0) + \alpha_{\varphi}\,\Vert u(t) - v_0 \Vert^2_V.
\end{eqnarray}
Since $\varphi(u, \cdot)$ is convex and lower semicontinuous for $u \in K$, it admits an affine minorant (cf. Proposition 5.2.25 in \cite{DMP1}), i.e.,  there are $l_{v_0} \in V^*$ and  $b_{v_0} \in \R$ such that $\varphi(v_0, v) \geqslant \langle l_{v_0},v \rangle_{V^* \times V} + b_{v_0}$ for all $v \in V$. Using this inequality, we deduce that $- \varphi (v_0, u) \leqslant  \Vert l_{v_0} \Vert_{V^*}\;\Vert u \Vert_V - b_{v_0}$ for all $u \in V,$ so
\begin{align}\label{oszfi}
\begin{split}
\varphi(v_0, v_0) &- \varphi (v_0, u(t)) \leqslant  \Vert l_{v_0} \Vert_{V^*}\;\Vert u(t) \Vert_V - b_{v_0} + \varphi(v_0, v_0) \leqslant\\ &\Vert l_{v_0} \Vert_{V^*}\;\Vert v_0 - u(t) \Vert_V + \Vert l_{v_0} \Vert_{V^*}\;\Vert v_0 \Vert_V + \vert b_{v_0} \vert + \vert \varphi(v_0, v_0) \vert.
\end{split}
\end{align}
Combining \eqref{oszfi*} and \eqref{oszfi}, we conclude that
\begin{align}\label{oszfiK}
\begin{split}
\varphi (u(t), v_0) &- \varphi(u(t),u(t)) \leqslant \Vert l_{v_0} \Vert_{V^*}\;\Vert v_0 - u(t) \Vert_V + \Vert l_{v_0} \Vert_{V^*}\;\Vert v_0 \Vert_V + \vert b_{v_0} \vert\\ &+ \vert \varphi(v_0, v_0)\vert + \alpha_{\varphi}\,\Vert u(t) - v_0 \Vert^2_V.
\end{split}
\end{align}
On the other hand, from Proposition 3.23(iii) in \cite{MOSBOOK}, the Cauchy-Schwartz inequality and the condition \eqref{wJ}(c), we obtain
\begin{align}\label{JOTY}
\begin{split}
J^0 (t, Mu(t);M(v_0 -u(t))) &= \mbox{max} \lbrace \langle \zeta(t), M(v_0 - u(t))\rangle_{X^* \times X} \ \vert \ \zeta(t) \in \partial J(t, Mu(t)) \rbrace\\
&\leqslant \Vert \partial J(t, Mu(t)) \Vert_{X^*}\;\Vert M(v_0 - u(t)) \Vert_X\\ &\leqslant (c_0(t) + c_1 \Vert M \Vert\,\Vert u(t) \Vert_V)\Vert M \Vert\,\Vert v_0 - u(t) \Vert_V.
\end{split}
\end{align}
Using conditions \eqref{wA}(b),(d) and estimates \eqref{oszfiK}, \eqref{JOTY} into the inequality \eqref{WZOR}, we see that
\begin{align*}
\begin{split}
(m_A - \alpha_{\varphi})\,\Vert &v_0 - u(t) \Vert^2_V \leqslant (a_0(t) + a_1\,\Vert v_0 \Vert_V)\,\Vert v_0 - u(t) \Vert_V + \Vert l_{v_0} \Vert_{V^*}\;\Vert v_0 - u(t) \Vert_V\\ &+ \Vert l_{v_0} \Vert_{V^*}\;\Vert v_0 \Vert_V + \vert b_{v_0} \vert + \vert \varphi(v_0, v_0)\vert + (c_0(t) + c_1 \Vert M \Vert\,\Vert u(t) \Vert_V)\Vert M \Vert\,\Vert v_0 - u(t) \Vert_V\\
&+ \Vert f(t)\Vert_{V^*}\,\Vert v_0 - u(t)\Vert_V.
\end{split}
\end{align*}
Hence, from the condition \eqref{WF}(b) and the elementary property, namely,  $x^2 \leqslant ax + b$ imply $x^2 \leqslant a^2 + b$ for $x,a,b \geqslant 0,$ we have
\begin{align*}
\begin{split}
\Vert v_0 - u(t) \Vert^2_V &\leqslant c^2\,(a_0(t) + a_1\,\Vert v_0 \Vert_V + \Vert l_{v_0} \Vert_{V^*}  + \Vert M \Vert\,c_0(t) + c_1 \Vert M \Vert^2\,\Vert u(t) \Vert_V + \Vert f(t)\Vert_{V^*})^2\\ &+ \Vert l_{v_0} \Vert_{V^*}\;\Vert v_0 \Vert_V + \vert b_{v_0} \vert + \vert \varphi(v_0, v_0)\vert.
\end{split}
\end{align*}
From this and inequality $\Vert u(t) \Vert^2 \leqslant 2\,\Vert u(t) - v_0\Vert_V^2 + 2\,\Vert v_0 \Vert_V^2,$ we conclude that
\begin{align*}
\begin{split}
\Vert u(t) \Vert^2_V &\leqslant 2\,[c^2\,(a_0(t) + a_1\,\Vert v_0 \Vert_V + \Vert l_{v_0} \Vert_{V^*}  + \Vert M \Vert\,c_0(t) + c_1 \Vert M \Vert^2\,\Vert u(t) \Vert_V + \Vert f(t)\Vert_{V^*})^2\\ &+ \Vert l_{v_0} \Vert_{V^*}\;\Vert v_0 \Vert_V + \vert b_{v_0} \vert + \vert \varphi(v_0, v_0)\vert] + 2\,\Vert v_0 \Vert_V^2.
\end{split}
\end{align*}
Thus, the inequality $\big(\sum\limits_{i=1}^{m}\,a_i \big)^2 \leqslant m\, \sum\limits_{i=1}^{m}\,a_i^2$ for $a_i \geqslant 0$ implies that
$$
\Vert u(t) \Vert_V^2 \leqslant c_1^2\,(a_0^2(t) + c_0^2(t) + \Vert f(t)\Vert_{V^*}^2 + c_2^2) + c_3,
$$
where $c_1, c_2, c_3 \geqslant 0$ are constants.
Integrating the last inequality over the interval $(0,T)$, we deduce that $\Vert u \Vert_{\v} \leqslant c$. Hence and the fact that $f \in \v^*$, we deduce that $u \in \v.$ 
The proof is finished.
\end{proof}
In the next problem, in contrast to Problem \ref{probabst1}, the convex function $\tilde \varphi$ depends on the three arguments which follows directly from the application (cf. Section \ref{s41}).
\begin{problem}\label{problem1a}
Find $u\in \v$ such that $u(t) \in K$ and 
\begin{align}\label{var-hemi}
\begin{split}
 \langle A(t, u(t)), v -u(t)\rangle_{V^* \times V} &+ \tilde \varphi ((\s u)(t), u(t),v) - \tilde \varphi ((\s u)(t), u(t), u(t))\\ &+ J^0 (t, Mu(t);M(v-u(t))) \geqslant \langle f(t), v-u(t) \rangle_{V^* \times V}
\end{split}
\end{align}
for all $v \in K$ and a.e. $t \in (0,T).$
\end{problem}
The inequality \eqref{var-hemi} represents a {\textit{variational-hemivariational inequality with  history-dependent operator}}. 

As before, we assume that the operators $\s, A$ and the functions $ J, f$ satisfy conditions \eqref{wS}, \eqref{wA}, \eqref{wJ} and \eqref{WF}(a), respectively. Additionally, we assume that the function
\begin{align}\label{fi01}
\left.
\begin{array}{l}
\tilde \varphi \colon V^* \times K \times K\longrightarrow \R\ \mbox{is such that}\\
\ \ {\rm (a)}\ 
\tilde \varphi (w,u, \cdot)\colon K \longrightarrow \R\ \mbox{is convex and lower semicontinuous on}\ K,\ \mbox{for all}\\ \hspace{1cm} w \in V^*,\  u \in K.\\
\ \ {\rm (b)}\ 
\mbox{there exists}\ \alpha_{\tilde \varphi} > 0\ \mbox{such that}\\
\hspace{1cm} \tilde \varphi (w_1, u_1, v_2) - \tilde \varphi (w_1,u_1, v_1) + \tilde \varphi (w_2,u_2, v_1) - \tilde \varphi (w_2,u_2, v_2)\\ 
\hspace{1cm} \leqslant \alpha_{\tilde \varphi} (\Vert u_1 - u_2 \Vert_V + \Vert w_1 - w_2 \Vert_V)\Vert v_1-v_2 \Vert_V\
\mbox{for all}\ w_1, w_2 \in V^*,\\
\hspace{1cm} u_1, u_2, v_1, v_2 \in K.
\end{array}
\right\}
\end{align}
\begin{twr}\label{twierdzenie11}
Under the assumptions \eqref{wS}, \eqref{wA}--\eqref{WF} and \eqref{fi01}, Problem \ref{problem1a} has a unique solution $u \in \v$.
\end{twr}
\begin{proof} 
Let $\eta \in \v^*$ be fixed and we consider the following auxiliary problem.
\begin{problem}\label{problem2a}
Find $u_{\eta}\in \v$ such that $u_{\eta}(t) \in K$ and
\begin{align*}
\begin{split}
\langle A(t, u_\eta(t)), v -u_\eta(t)\rangle_{V^* \times V} &+ \tilde \varphi (\eta(t),u_\eta(t),v) - \tilde \varphi (\eta(t),u_\eta(t), u_\eta(t)) \\ &+ J^0 (t, Mu_\eta(t);M(v-u_\eta(t))) \geqslant \langle  f(t), v-u_{\eta}(t) \rangle_{V^* \times V}
\end{split}
\end{align*}
for all $v \in K$ and a.e. $t \in (0,T).$
\end{problem}
Let $\phi_\eta \colon K \times K \longrightarrow \R$ be defined by $\phi_\eta(w,v)= \tilde \varphi (\eta(t), w, v)$ for all $w,v \in K$ and for a.e. $t \in (0,T)$. We show that function $\phi_\eta$ satisfies \eqref{fi0}. It is easy to see that function $\phi_\eta (u,\cdot)$ satisfies \eqref{fi01}(a)  for a.e.\;$t \in (0,T)$ and for all $u \in K$. Moreover, using \eqref{fi01}(b), we infer that
\begin{align}
\begin{split}
&\phi_\eta(u_1,v_2) - \phi_\eta(u_1,v_1) + \phi_\eta(u_2,v_1) - \phi_\eta(u_2,v_2) = \tilde \varphi (\eta(t),u_1,v_2)\\ &- \tilde \varphi (\eta(t),u_1,v_1) + \tilde \varphi (\eta(t),u_2,v_1) - \tilde \varphi (\eta(t),u_2,v_2) \leqslant \alpha_{\tilde \varphi}\,\Vert u_1 - u_2 \Vert_V \Vert v_1 - v_2 \Vert_V
\end{split}
\end{align}
${\rm for\ all}\ u_1, u_2, v_1, v_2 \in K,$ and a.e. $t \in (0,T).$ So, the condition \eqref{fi0}(b) holds with $\alpha_{\varphi} = \alpha_{\tilde \varphi}.$

Hence and from Theorem \ref{twierdzenie1}, we deduce that Problem \ref{problem2a} has the unique solution $u_{\eta} \in \v$.
Next, we define the operator 
$\Lambda: \v^* \longrightarrow \v^*$ by
$
\Lambda \eta = \s u_{\eta} \ {\rm for\ all} \ \eta \in \v^*,
$
where $u_{\eta} \in \v$ is the solution to Problem \ref{problem2a}. 
\begin{klaim}\label{lematlambda}
The operator $\Lambda$ has a unique fixed point $\eta^* \in \v^*$.
\end{klaim}
Let $\eta_1, \eta_2 \in \v^*,\ t \in (0,T)$ and let $u_i= u_{\eta_i} \in \v$ for $i=1,2,$ be the corresponding solutions to Problem \ref{problem2a}. We put into the inequality in Problem \ref{problem2a}, $v=u_2(t) - u_1(t)$ and $v= u_1(t) - u_2(t),$ respectively. Thus,
\begin{align*}
\begin{split}
\langle A(t, u_1(t)), u_2(t) -u_1(t)\rangle_{V^* \times V} &+ \tilde \varphi (\eta_1(t),u_1(t),u_2(t)) - \tilde \varphi (\eta_1(t),u_1(t), u_1(t)) \\ &+ J^0 (t, Mu_1(t);M(u_2(t)-u_1(t))) \geqslant \langle  f(t), u_2(t)-u_1(t) \rangle_{V^* \times V}
\end{split}
\end{align*}
and
\begin{align*}
\begin{split}
\langle A(t, u_2(t)), u_1(t) -u_2(t)\rangle_{V^* \times V} &+ \tilde \varphi (\eta_2(t),u_2(t),u_1(t)) - \tilde \varphi (\eta_2(t),u_2(t), u_2(t)) \\ &+ J^0 (t, Mu_2(t);M(u_1(t)-u_2(t))) \geqslant \langle  f(t), u_1(t)-u_2(t) \rangle_{V^* \times V}.
\end{split}
\end{align*}
Adding obtained inequalities, we have
\begin{align*}
\begin{split}
&\langle A(t, u_1(t)) - A(t, u_2(t)), u_1(t) - u_2(t) \rangle_{V^* \times V} - \Big( J^0 (t, Mu_1(t);M(u_2(t)-u_1(t)))\\ &+ J^0 (t, Mu_2(t);M(u_1(t)-u_2(t)))\Big) \leqslant \tilde \varphi(\eta_1(t), u_1(t),u_2(t)) - \tilde \varphi(\eta_1(t),u_1(t),u_1(t))\\ &+ \tilde \varphi(\eta_2(t),u_2(t),u_1(t)) - \tilde \varphi(\eta_2(t),u_2(t),u_2(t)).
\end{split}
\end{align*}
Using \eqref{wA}(b), \eqref{wJ}(e) and \eqref{fi01}(b), we get
\begin{align*}
\begin{split}
m_A\,\Vert u_1(t) - u_2(t) \Vert^2_V &- \Big( \alpha_{\tilde \varphi}\, \Vert u_1(t) - u_2(t) \Vert^2_V + m_J\Vert M \Vert^2\,\Vert u_1(t) - u_2(t) \Vert^2_V \Big)\\ &\leqslant \alpha_{\tilde \varphi}\Vert \eta_1(t) - \eta_2(t) \Vert_{V^*} \Vert u_1(t) - u_2(t) \Vert_{V}.
\end{split}
\end{align*}
Hence, by the condition \eqref{WF}(b) with $\alpha_{\varphi} = \alpha_{\tilde \varphi}$, we obtain
\begin{equation*}
\Vert u_1(t) - u_2(t) \Vert_V \leqslant c\, \Vert \eta_1(t) - \eta_2(t) \Vert_{V^*}
\end{equation*}
\noindent
which together with the inequality (cf. \eqref{wS})
$$
\Vert (\Lambda \eta_1)(t) - (\Lambda \eta_2)(t) \Vert_{V^*} = \Vert (\s u_1)(t) - (\s u_2)(t) \Vert_{V^*} \leqslant L_{\s}\, \int\limits_{0}^{t}\, \Vert u_1(s) - u_2(s) \Vert_{V}\, ds
$$
imply that
$
\Vert (\Lambda \eta_1)(t) - (\Lambda \eta_2)(t) \Vert_{V^*}\leqslant cL_{\s} \, \int\limits_{0}^{t}\, \Vert \eta_1(s) - \eta_2(s) \Vert_{V^*}\, ds.
$
From the last inequality and the H\"older inequality, we conclude that
$$
\Vert (\Lambda \eta_1)(t) - (\Lambda \eta_2)(t) \Vert_{V^*}^2\leqslant c\, \int\limits_{0}^{t}\, \Vert \eta_1(s) - \eta_2(s) \Vert_{V^*}^2\, ds \ \ {\rm for\ a.e.}\;t \in (0,T).
$$
Applying Lemma \ref{lematkulig}, we deduce that there exists  a unique $\eta^* \in \v^*$ such that $\Lambda \eta^* = \eta^*,$ which concludes the proof of the claim.

Now, we continue the proof of Theorem \ref{twierdzenie11}.
\\
{\bf Existence.} Let $\eta^* \in  \v^*$ be the fixed point of the operator $\Lambda$ (cf. Claim \ref{lematlambda}). We put $\eta = \eta^*$ in Problem \ref{problem2a}  and since $\eta^* = \Lambda \eta^* = \s u_{\eta^*}$, we see that $u_{\eta^*} \in \v$ is a solution to Problem \ref{problem1a}.\\ 
{\bf Uniqueness.} Here, we use the Gronwall-type argument. Let $u_1, u_2 \in \v$ be solutions to Problem \ref{problem1a} and $t \in (0,T).$ Then, proceeding similarly as in the proof of Theorem \ref{twierdzenie1}, we see that
\begin{align*}
\begin{split}
\langle A(t, &u_1(t)) - A(t, u_2(t)), u_1(t) -u_2(t)\rangle_{V^* \times V} -\Big( J^0 (t, Mu_1(t);M(u_2(t)-u_1(t)))\\ &+ J^0 (t, Mu_2(t);M(u_1(t)-u_2(t))) \Big) \leqslant \tilde \varphi ((\s u_1)(t),u_1(t),u_2(t)) - \tilde \varphi ((\s u_1)(t),u_1(t), u_1(t))\\ &+ \tilde \varphi ((\s u_2)(t),u_2(t),u_1(t)) - \tilde \varphi ((\s u_2)(t),u_2(t), u_2(t)) .
\end{split}
\end{align*}
Using conditions \eqref{wA}(b), \eqref{wJ}(e) and \eqref{fi01}(b), we get
\begin{align*}
\begin{split}
m_A\,\Vert u_1(t) - u_2(t) \Vert^2_{V} &- (\alpha_{\tilde \varphi} + m_J\Vert M \Vert^2) \Vert u_1(t) - u_2(t) \Vert^2_V \\ &\leqslant \alpha_{\tilde \varphi}\,\Vert (\s u_1)(t) - (\s u_2)(t) \Vert_{V^*} \Vert u_1(t) - u_2(t) \Vert_V.
\end{split}
\end{align*}
Next, from \eqref{wS} and \eqref{WF}(b), we have
$
\Vert u_1(t) - u_2(t) \Vert_V \leqslant c\, \int\limits_{0}^{t}\, \Vert u_1(s) - u_2(s) \Vert_V\,ds$ for a.e. $t \in (0,T).
$
Using the Gronwall inequality, we obtain
$
\Vert u_1(t) - u_2(t) \Vert_V = 0$ for a.e. $ t \in (0,T)$,
which implies that $u_1(t) = u_2(t)$ for a.e. $t \in (0,T).$ The proof
of the theorem is complete.
\end{proof}
\section{Quasistatic elastic-viscoplastic contact problem with normal damped response, unilateral constraint and memory term}\label{s41}
In this section we use the results obtained in Section \ref{s31} into the study the elastic-viscoplastic contact problem with normal damped response, unilateral constraint and memory term.

The physical setting is as follows. An elastic-viscoplastic body occupies a bounded domain $\Omega \subset \R^d$, where $d=2,3$ in applications. The boundary $\Gamma$ of the domain $\Omega$ is Lipschitz
continuous and it is partitioned into three disjoint measurable parts $\Gamma_1, \Gamma_2$ and $\Gamma_3$ with $\mbox{meas}(\Gamma_1) > 0.$ The body is subject to the action of body forces of density $f_0$ and surface tractions of density $f_2$ which act on $\Gamma_2.$ We assume that the body is clamped on $\Gamma_1$ and it is in contact on $\Gamma_3$ with a rigid foundation. Furthermore the mechanical process is quasistatic and we study it in the time interval $[0,T]$ with $T>0.$

We use the notation $\R^d$ and $\es^d$ for the $d-$dimensional real linear space and the space of second order symmetric tensors on $\R^d$, respectively, which are equipped with the following canonical inner products and  norms 
$$
u \cdot v =  u_{i} v_{i}, \quad \Vert v \Vert_{\mathbb{R}^{d}} = (v \cdot v)^{\frac{1}{2}} \quad \mathrm{for \, all} \quad  u= (u_{i}), \, v = (v_{i}) \in \mathbb{R}^{d},
$$
$$
\sigma : \tau = \sigma_{ij} \tau_{ij}, \quad \Vert \tau \Vert_{\mathbb{S}^{d}} = (\tau : \tau)^{\frac{1}{2}} \quad \mathrm{for \, all} \quad  \sigma= (\sigma_{ij}), \, \tau = (\tau_{ij}) \in \mathbb{S}^{d},
$$
where the indices $i$ and $j$ run between $1$ and $d$.
Let us add, that the summation convention over
repeated indices is used.  Let $u' = \frac{\partial u}{\partial t}$ represent the  velocity field and let  ${\rm Div} \sigma = (\sigma_{ij,j})$ be the divergence operator. We use the standard notation for the Lebesgue and Sobolev spaces and we introduce the following Hilbert spaces
$$H = L^2(\Omega;\R^d) = \lbrace v = (v_i)\ \vert\ v_i \in L^2(\Omega),\ 1 \leqslant i \leqslant d \rbrace, 
$$
$$
\h =L^2(\Omega; \es^d) = \lbrace \tau = (\tau_{ij})\ \vert\ \tau_{ij} = \tau_{ji} \in L^2(\Omega),\ 1 \leqslant i,\ j \leqslant d \rbrace,\quad
\h_1 = \lbrace \tau \in \h \ \vert \ {\rm Div}\tau \in H \rbrace.
$$
It is worth mentioning, that the Hilbert space, presented above, are equipped with the canonical inner products
$$
( u, v )_{H} = \int\limits_{\Omega}^{} u \cdot v \ dx,
\quad
( \sigma, \tau )_{\mathcal{H}} = \int\limits_{\Omega}^{} \sigma : \tau \ dx,
\quad
( \sigma,\tau )_{\h_1} = ( \sigma, \tau )_{\h} + \big( {\rm Div} \sigma, {\rm Div} \tau \big)_{H}$$
and the associated norms 
$$
\Vert v \Vert_{H} = \Big(\int\limits_{\Omega}^{}\, (\Vert v(x) \Vert_{\R^d})^2\,dx \Big)^{\frac{1}{2}},\ \   
\Vert \tau \Vert_{\h} = \Big(\int\limits_{\Omega}^{}\,\Vert \tau(x) \Vert_{\es^d}^2\, dx \Big)^{\frac{1}{2}},\ \ \Vert \tau \Vert_{\h_1} = \Vert \tau \Vert_{\h} + \Vert {\rm Div}\,\tau \Vert_H, $$
respectively. 

We consider also the real Hilbert space for the displacement
$$
V = \lbrace v \in H^1(\Omega;\R^d)\ \vert \ v=0 \ {\rm a.e.\ on}\ \Gamma_1\ \mbox{and}\ v_\nu = 0\ \mbox{a.e. on}\ \Gamma_3 \rbrace.
$$
This space is endowed with the inner product and the associated norm given by
\begin{align*}
( u,v )_V = ( \varepsilon (u), \varepsilon(v) )_{\h}
\quad
{\rm and} \quad
\Vert v \Vert_V = \Vert \varepsilon(v) \Vert_{\h},
\end{align*}
where
$
\varepsilon (u) = (\varepsilon_{ij} (u))$ such that $\varepsilon_{ij} (u) = \frac{1}{2} \Big( \frac{\partial u_{i}}{\partial x_{j}} + \frac{\partial u_{j}}{\partial x_{i}}\Big)
$
is the deformation operator.
Additionally, the inequality  $\Vert v \Vert_{L^2(\Gamma_3;\R^d)} \leqslant c_0\, \Vert v \Vert_V$ holds for all $v \in V$, where $c_0$ is a constant which depends on $\Omega,\ \Gamma_1$ and $\Gamma_3$.  

Assume that $\nu$ denote the outward unit normal vector
on $\Gamma, v \in H^{1}(\Omega;\R^d)$ and $\sigma$ is a regular function. Thus, 
the normal and tangential components of the displacement field (stress field) on the boundary $\Gamma,$ are defined by
$
v_{\nu} = v \cdot \nu,\ v_{\tau} = v - v_{\nu} \cdot \nu\ 
\big(\sigma_{\nu} = (\sigma \nu) \cdot \nu,\ \sigma_{\tau} = \sigma \nu - \sigma_{\nu} \cdot \nu\big)$.
In order to derive variational formulations of the contact problems we will use the Green formula and the decomposition formula which are presented below.
\begin{equation}\label{GF}
\big( \sigma, \varepsilon (v) \big)_{\h} + \big( {\rm Div} \sigma, v \big)_{H} = \int\limits_{\Gamma}^{} \sigma \nu \cdot v\, d \Gamma\quad {\rm for\ all}\ v \in H^1(\Omega;\R^d).
\end{equation}
\begin{equation}\label{decom}
\sigma \nu \cdot v = \sigma_{\nu} v_{\nu} + \sigma_{\tau} \cdot v_{\tau}. 
\end{equation}
For simplicity,  we will write $v$ instead of $\gamma v,$ where $\gamma$ denotes the trace of $v$ on the boundary $\Gamma$. For simplicity, we use the following notation $Q= \Omega \times (0,T)$ and $\Sigma_i = \Gamma_i \times (0,T)$ for $i=1,2,3.$

We study the elastic-viscoplastic contact problem which classical formulation is the following.
\begin{problem}\label{model1}
Find a displacement field $u \colon Q \longrightarrow \R^d$ and a stress field $\sigma\colon Q \longrightarrow \es^d$ such that 
\begin{eqnarray}
 \sigma (t) = \a(t, \varepsilon(u'(t))) + \b (t,\varepsilon(u(t))) + \int\limits_{0}^{t}\, \g(s, \sigma(s) - \a(s, \varepsilon(u'(s))), \varepsilon(u(s)))\,ds \ \ \mbox{in}\ \ Q,\qquad
\label{constitutivelaw}\\
{\rm Div}\, \sigma(t) + f_0(t) = 0  \ \ \mbox{in}\ \ Q,\qquad
\label{equationofmotion}\\
u(t) = 0 \ \ \mbox{on}\ \ \Sigma_1,\qquad
\label{boundary1}
\end{eqnarray}
\begin{eqnarray}
\sigma(t) \nu = f_2(t)\ \ \mbox{on}\ \ \Sigma_2,\hspace{-10.5cm}
\label{boundary2}\\
\hspace{-1.7cm}
- \sigma_{\tau}(t) \in \partial j_{\tau}(t, u'_{\tau}(t)) \ \ \mbox{on}\ \ \Sigma_3,\hspace{-10.5cm}
\label{boundary3}
\end{eqnarray}
\begin{equation}\label{signorini}
\hspace{3.15cm}
\left.
\begin{array}{l}
\hspace{-0.2cm} 
u'_{\nu}(t) \leqslant g,\quad \sigma_{\nu}(t) + p(u'_{\nu}(t)) + \int\limits_{0}^{t}\,b(t-s) (u^{'}_{\nu})^{+}(s)\,ds \leqslant 0 \\ [2mm]
\hspace{-0.4cm} 
(u'_{\nu}(t) - g)\Big(\sigma_{\nu}(t) + p(u'_{\nu}(t)) + \int\limits_{0}^{t}\,b(t-s) (u^{'}_{\nu})^{+}(s)\,ds \Big)= 0
\end{array}
\right\}\ \ \mbox{on}\ \ \Sigma_3,
\end{equation}
\begin{eqnarray}
\hspace{11.9cm}
u(0) = u_0\ \ \mbox{in}\ \ \Omega.\qquad
\label{po}
\end{eqnarray} 
Let us note that equation \eqref{constitutivelaw} is the elastic-viscoplastic constitutive law in which $\a$ is the viscosity operator, $\b$ is the elasticity operator and $\g$ is the viscoplasticity operator. The equilibrium equation is presented by \eqref{equationofmotion}. The displacement and the traction boundary conditions are expressed by \eqref{boundary1} and \eqref{boundary2}, respectively. The conditions \eqref{boundary3} and \eqref{signorini} are the friction law and the contact condition with normal compliance, unilateral constraint and memory term. The law \eqref{signorini} without memory term, is considered in \cite{BDS}. Here, $p$ and $b$ represent given contact functions. Finally, \eqref{po} is the initial condition and $u_0$ denotes the initial displacement. 
\end{problem}
In the study of Problem \ref{model1}, we need the following assumptions.
\begin{align}\label{wkA}
\left.
\begin{array}{l}
\a: Q \times \es^d \longrightarrow \es^d\ \mbox{is an operator such that}\\
\ \ {\rm (a)}\  
\a (\cdot,\cdot, \varepsilon)\ \mbox{is measurable on}\ Q\ \mbox{for all}\ \varepsilon \in \es^d.\\
\ \ {\rm (b)}\ 
\a(x,t,\cdot)\ \mbox{is strongly monotone, i.e., there exists}\ m_{\a} > 0\ \mbox{such that}\\
\hspace{1cm} \big( \a (x,t, \varepsilon_1) - \a (x,t, \varepsilon_2)\big):( \varepsilon_1 - \varepsilon_2) \geqslant m_{\a} \Vert \varepsilon_1 - \varepsilon_2 \Vert^2_{\es^d} \\
\hspace{1cm} \mbox{for all}\ \varepsilon_1, \varepsilon_2 \in \es^d\ \mbox{and a.e.}\ (x,t) \in Q.\\
\ \ {\rm (c)}\ 
\a(x,t,\cdot)\ \mbox{is continuous on}\ \es^d,\ \mbox{for a.e.}\ (x,t) \in Q.\\
\ \ {\rm (d)}\ 
\Vert \a(x,t, \varepsilon) \Vert_{\es^d} \leqslant \overline{a}_0 (x,t) + \overline{a}_1 \Vert \varepsilon \Vert_{\es^d}\ \mbox{for all}\ \varepsilon \in \es^d\ \mbox{and a.e.}\ (x,t) \in Q\\
 \hspace{1cm} \mbox{with}\ \overline{a}_0 \in L^2 (Q), \overline{a}_0 \geqslant 0\ \mbox{and}\ \overline{a}_1 >0.\\
 \ \ {\rm (e)}\ 
\mbox{there exists}\ \alpha_{\a} > 0\ \mbox{such that}\ 
 \a(x,t,\varepsilon): \varepsilon \geqslant \alpha_{\a}\Vert \varepsilon \Vert^2_{\es^d} \\
\hspace{1cm} {\rm for\ all} \ \varepsilon \in \es^d \ {\rm and\ a.e.} \ (x,t) \in Q.
\end{array}
\right\}
\end{align}
\begin{align}\label{wB}
\left.
\begin{array}{l}
\b: Q \times \es^d \longrightarrow \es^d\ \mbox{is an operator such that}\\
\ \ {\rm (a)}\ 
\b (\cdot, \cdot,\varepsilon)\ {\rm is\ measurable\ on}\ Q\ {\rm for\ all}\ \varepsilon \in \es^{d}$ and $\b (\cdot,\cdot,0) \in L^2 (Q ; \es^d).\\
\ \ {\rm (b)}\ 
\Vert \b (x, t, \varepsilon_1) - \b (x,t, \varepsilon_2)\Vert_{\es^{d}} \leqslant L_{\b}\ \Vert \varepsilon_1 - \varepsilon_2 \Vert_{\es^d}\ {\rm for\ all}\ \varepsilon_1, \varepsilon_2 \in \es^{d},\\ 
\hspace{1cm} {\rm a.e.}\ (x, t) \in Q\  {\rm with}\ L_{\b} > 0.
\end{array}
\right\}
\end{align}
\begin{align}\label{wC}
\left.
\begin{array}{l}
\g: Q \times \es^d \times \es^d \longrightarrow \es^d\ \mbox{is an operator such that}\\
\ \ {\rm (a)}\ 
\g (\cdot, \cdot,\sigma, \varepsilon)\ {\rm is\ measurable\ on}\ Q\ {\rm for\ all}\ \sigma, \varepsilon \in \es^{d}\ \mbox{and}\\ \hspace{1cm} \g (\cdot,\cdot,0,0)\ \in L^2 (Q ; \es^d).\\
\ \ {\rm (b)}\ 
\Vert \g (x, t, \sigma_1, \varepsilon_1) - \g (x,t, \sigma_2, \varepsilon_2)\Vert_{\es^{d}} \leqslant L_{\g}\ (\Vert \sigma_1 - \sigma_2 \Vert_{\es^d} + \Vert \varepsilon_1 - \varepsilon_2 \Vert_{\es^d})\ \mbox{for}\\
\hspace{1cm} {\rm all}\ \sigma_1, \sigma_2, \varepsilon_1, \varepsilon_2 \in \es^{d},\ {\rm a.e.}\ (x, t) \in Q\  {\rm with}\ L_{\g} > 0.
\end{array}
\right\}
\end{align}
\begin{align}\label{wj}
\left.
\begin{array}{l}
j_{\tau}: \Sigma_{3} \times \R^d \longrightarrow \R\ \mbox{is such that}\hspace{1cm}\\
\ \ {\rm (a)}\ 
j_{\tau} (\cdot, \cdot,\xi)\ {\rm is\ measurable\ on}\ \Sigma_3\ {\rm for\ all}\ \xi \in \R^d,\ \mbox{and there exists}\\ \hspace{1cm} {\rm e} \in L^2 (\Gamma_3; \R^d)\ \mbox{such that}\  j_{\tau}(\cdot, \cdot, {\rm e}(\cdot)) \in L^1(\Sigma_3).\hspace{1cm}\\
\ \ {\rm (b)}\ 
j_{\tau} (x,t,\cdot)\ \mbox{is locally Lipschitz on}\ \R^d\ \mbox{for a.e.}\ (x,t) \in \Sigma_3.\hspace{1cm}\\
\ \ {\rm (c)}\ 
\Vert \partial j_{\tau} (x,t,\xi) \Vert_{\R^d} \leqslant \overline{c}_0(t) + \overline{c}_1 \Vert \xi \Vert_{\R^d}\ \mbox{for all}\ \xi \in \R^d,\ {\rm a.e.}\ (x,t) \in \Sigma_3,\hspace{1cm}\\
\hspace{1cm} c_0 \in L^2(0,T)\ \mbox{with}\ \overline{c}_0,  \overline{c}_1 \geqslant 0.\hspace{1cm}\\
\ \ {\rm (d)}\ 
\mbox{there exists}\ \alpha_j > 0\ \mbox{such that} \hspace{1cm}\\
\hspace{1cm} j_\tau^0(x,t,\xi_2;\xi_1-\xi_2) + j_\tau^0 (x,t, \xi_1; \xi_2-\xi_1) \leqslant \alpha_j \Vert \xi_1-\xi_2 \Vert^2_{\R^d} \mbox{for all}\hspace{1cm}\\ 
\hspace{1cm} \xi_1, \xi_2 \in \R^d,\ \mbox{and a.e.}\ (x,t) \in \Sigma_3.\hspace{1cm}\\
\ \ {\rm (e)}\ 
j_{\tau}(x,t,\cdot)\ \mbox{or}\ - j_{\tau}(x,t,\cdot)\ \mbox{is regular on}\ \R^d,\ \mbox{for a.e.}\ (x,t) \in \Sigma_3.\hspace{1cm}
\end{array}
\right\}
\end{align}
\begin{align}\label{wpa}
\left.
\begin{array}{l}
p\colon \Gamma_3 \times \R \longrightarrow \R_+\ \mbox{is such that}\hspace{1.2cm}\\
\ \ {\rm (a)}\ 
p(\cdot,r)\ \mbox{is measurable on}\ \Gamma_3\ \mbox{for all}\ r \in \R\ \mbox{and}\ p(\cdot,0) \in L^2(\Gamma_3).\hspace{1.2cm}\\
\ \ {\rm (b)}\ 
\mbox{there exists}\ L_p > 0\ \mbox{such that}\ 
\vert p(x,r_1) - p(x,r_2) \vert \leqslant L_p \vert r_1 - r_2 \vert
\hspace{1.2cm}\\
\hspace{1cm} \mbox{for all}\ r_1,r_2 \in \R,\ \mbox{a.e.}\ x \in \Gamma_3.\hspace{1.2cm}
\end{array}
\right\}
\end{align}
\begin{equation}\label{gap}
g \in L^{\infty}(\Gamma_3), \quad g > 0.
\end{equation}
\begin{equation}\label{wopb}
b \in L^1 (0,T;L^{\infty}(\Gamma_3)).
\end{equation}
\begin{equation}\label{wf1f2}
f_0 \in L^2(0,T;L^2(\Omega;\R^d)), \quad f_2 \in L^2(0,T;L^2(\Gamma_2;\R^d)).
\end{equation}
\begin{equation}\label{uzero}
u_0 \in V
\end{equation}
\par The concrete example of the function $j_\tau$ which satisfies condition \eqref{wj} is as follows $j_\tau (\xi)= \Vert \xi \Vert_{\R^d}$ for all $\xi \in \R^d$. Here, for simplicity, we omit the dependence on variables $(x,t)$. The subdifferential of function $j_\tau$ has the form
\begin{displaymath}
\partial j_\tau(\xi) = \left\{ \begin{array}{ll}
\overline{B}(0, 1) & \textrm{if $\xi = 0$}\\
\frac{\xi}{\Vert \xi \Vert_{\R^d}} & \textrm{if $\xi \neq 0$}
\end{array} \right.
\end{displaymath}
for all $\xi \in \R^d,$ where $\overline{B}(0, 1)$ denotes the closed unit ball in $\R^d.$ 
Note that the function $j_\tau$ is convex and regular. We see that \eqref{wj} holds with $\overline{c}_0 = 1, \overline{c}_1 = 0$ and $\alpha_j = 0$ (cf. Section 7.4 in \cite{MOSBOOK}).

Now, we provide the variational formulation of Problem \ref{model1}. To this end, we introduce the set of admissible displacement fields defined by
\begin{equation}\label{zbioru}
K = \{ v \in V\ \vert \  v_{\nu} \leqslant g\ {\rm a.e.\ on}\ \Gamma_3 \}.
\end{equation}
Assume that $(u,\sigma)$ are sufficiently smooth functions which solve \eqref{constitutivelaw}--\eqref{po}. Let $t \in (0,T)$ be fixed and $v \in K.$ We use the Green formula \eqref{GF} and the equation \eqref{equationofmotion} to obtain
$$
\int\limits_{\Omega}^{}\,\sigma(t) : (\varepsilon(v) - \varepsilon(u'(t)))\,dx = \int\limits_{\Omega}^{}\,f_0(t)\cdot (v-u'(t))\,d x + \int\limits_{\Gamma}^{}\,\sigma(t)\nu \cdot (v - u'(t))\,d\Gamma.
$$
Using \eqref{boundary1}, \eqref{boundary2} and the decomposition formula \eqref{decom}, we get
\begin{align}\label{wzorg}
\begin{split}
&\int\limits_{\Omega}^{}\,\sigma(t) : (\varepsilon(v) - \varepsilon(u'(t)))\,d x = \int\limits_{\Omega}^{}\,f_0(t)\cdot (v-u'(t))\,dx \\ &+ \int\limits_{\Gamma_2}^{}\,f_2(t)\cdot (v-u'(t))\,d \Gamma + \int\limits_{\Gamma_3}^{}\,\Big(\sigma_{\nu}(t)(v_{\nu} - u'_{\nu}(t)) + \sigma_{\tau}(t)\cdot(v_{\tau} - u'_{\tau}(t))\Big)\, d \Gamma.
\end{split}
\end{align}
From \eqref{signorini}, we see that
\begin{align}\label{sigmani}
\begin{split}
\sigma_{\nu}(t)(v_{\nu} - u'_{\nu}(t)) &= \Big( \sigma_{\nu}(t) + p(u'_{\nu}(t)) + \int\limits_{0}^{t}\,b(t-s) (u^{'}_{\nu})^{+}(s)\,ds \Big)(v_{\nu} - g)\\
&+ \Big( \sigma_{\nu}(t) + p( u'_{\nu}(t)) + \int\limits_{0}^{t}\,b(t-s) (u^{'}_{\nu})^{+}(s)\,ds \Big) (g - u'_{\nu}(t))\\
&- \Big( p(u'_{\nu}(t)) + \int\limits_{0}^{t}\,b(t-s) (u^{'}_{\nu})^{+}(s)\,ds \Big)(v_{\nu} - u'_{\nu}(t)) \quad {\rm on}\ \Gamma_3.
\end{split}
\end{align}
From the contact condition \eqref{signorini} and the definition  of set $K$ (cf. \eqref{zbioru}), we have
\begin{align*}
\begin{split}
\sigma_{\nu}(t)(v_{\nu} - u'_{\nu}(t)) \geqslant - \Big( p(u'_{\nu}(t)) + \int\limits_{0}^{t}\,b(t-s) (u^{'}_{\nu})^{+}(s)\,ds \Big) (v_{\nu} - u'_{\nu}(t)) \quad {\rm on}\ \Gamma_3,
\end{split}
\end{align*}
and
\begin{align}
\begin{split}
\int\limits_{\Gamma_3}^{}\,\sigma_{\nu}(t)(v_{\nu} &- u'_{\nu}(t))\, d \Gamma \geqslant\\ &- \int\limits_{\Gamma_3}^{}\,\Big( p(u'_{\nu}(t)) + \int\limits_{0}^{t}\,b(t-s) (u^{'}_{\nu})^{+}(s)\,ds \Big)(v_{\nu} - u'_{\nu}(t))\,d\Gamma.
\end{split}
\end{align}
 The definition of the Clarke subdifferential and the boundary condition \eqref{boundary3} imply that
\begin{align}\label{sigmatau}
\int\limits_{\Gamma_3}^{}\,\sigma_{\tau}(t)\cdot(v_{\tau} - u'_{\tau}(t))\, d \Gamma \geqslant \int\limits_{\Gamma_3}^{}\,j^0_{\tau}(t,u'_{\tau}(t);v_{\tau} - u'_{\tau}(t))\,d \Gamma.
\end{align}
Using the definition of the space $V$, we note that
$$
v \longmapsto \int\limits_{\Omega}^{}\,f_0(t)\cdot v\, d x + \int\limits_{\Gamma_2}^{}\,f_2(t) \cdot v \, d \Gamma \quad {\rm for\ a.e.}\ t \in (0,T)
$$
is a linear, continuous functional on $V.$ Therefore, we may apply
the Riesz representation theorem to define the function $f \colon (0,T) \longrightarrow V^*$ by
\begin{align}\label{deff}
\langle f(t), v \rangle_{V^* \times V} = ( f_0(t),v )_H + ( f_2(t), v  )_{L^2(\Gamma_2; \R^d)}
\end{align}
for all $v \in V$ and a.e. $t \in (0,T).$
Combining \eqref{constitutivelaw} and \eqref{wzorg}--\eqref{deff}, we obtain the following variational formulation of Problem \ref{model1}.
\begin{problem}\label{model2}
Find $u \in \w$ such that $u(t) \in K,\, \sigma \in L^2(0,T;\h)$  and
\begin{align*}
\begin{split}
&\sigma(t) = \a(t,\varepsilon(u'(t)) + \b(t,\varepsilon(u(t))+\int\limits_{0}^{t}\, \g(s, \sigma(s) - \a(s, \varepsilon(u'(s))), \varepsilon(u(s)))\,ds,\ \mbox{a.e.}\ t \in (0,T)\\
&(\sigma(t), \varepsilon(v) - \varepsilon(u'(t)))_{\h} + \int\limits_{\Gamma_3}{}\,p(u'_{\nu}(t))(v_{\nu} - u'_{\nu}(t))\,d \Gamma + \int\limits_{\Gamma_3}\,\Big( \int\limits_{0}^{t}\,b(t-s) (u^{'}_{\nu})^{+}(s)\,ds \Big)(v_{\nu} - u'_{\nu}(t))\,d \Gamma\\
&+ \int\limits_{\Gamma_3}^{}\,j^0_{\tau}(t,u'_{\tau}(t);v_{\tau} - u'_{\tau}(t))\,d \Gamma \geqslant \langle f(t), v - u'(t)\rangle_{V^* \times V}
\end{split}
\end{align*}
for all $v \in K$ and a.e. $t \in (0,T)$ with $u(0) = u_0.$
\end{problem}
The existence and uniqueness result for Problem \ref{model2} is the following.
\begin{twr}\label{twierdzeniemodel1}
Under the assumptions \eqref{wkA}--\eqref{uzero} and
\begin{equation}\label{nier4}
m_{\a} > \mbox{\rm max}\,\lbrace 1, L_P \rbrace + \alpha_j \Vert \gamma \Vert^2,\ \ \alpha_{\a} > 2\,\alpha_j\,\Vert \gamma \Vert^2
\end{equation}
Problem \ref{model2} has a unique solution.
\end{twr}
\begin{proof} The proof of this theorem will be carried out in two steps.\\
{\bf{Step 1.}} We need the following auxiliary result.
\begin{lemat}\label{OPHISIGMA}
Assume that \eqref{wB} and \eqref{wC} hold. Then, for all $u \in \v,$ there exists a unique function $\sigma^{I}(u) \in L^2(0,T;\h)$ such that
\begin{equation}\label{SI}
\sigma^I(u(t)) = \int\limits_{0}^{t}\,\g\big(s,\b(t,\varepsilon(u(s))) + \sigma^I(u(s)), \varepsilon(u(s))\big)\,ds
\end{equation}
for a.e. $t\in (0,T).$ Moreover, if $u_1, u_2 \in \v,$ then
$$
\Vert \sigma^I(u_1)(t) - \sigma^I (u_2)(t) \Vert_{\h} \leqslant L_{\sigma^I}\,\int\limits_{0}^{t}\,\Vert u_1(s) - u_2(s) \Vert_V\,ds
$$
for a.e. $t \in (0,T)$ with $L_{\sigma^I}>0.$
\end{lemat}
\begin{proof}
The proof of the lemma is presented in Lemma 6.1 in \cite{CMOS}.
\end{proof}

In order to formulate an equivalent form of Problem \ref{model2}, we use Lemma \ref{OPHISIGMA}. We consider the following intermediate problem.
\begin{problem}\label{model2I}
Find $u \in \w$ such that $u(t) \in K,\ \sigma \in L^2(0,T;\h)$  and
\begin{align*}
\begin{split}
&\sigma(t) = \a(t,\varepsilon(u'(t)) + \b(t,\varepsilon(u(t))+ \sigma^I(u(t))\ \mbox{a.e.}\ t \in (0,T)\\
&(\sigma(t), \varepsilon(v) - \varepsilon(u'(t)))_{\h} + \int\limits_{\Gamma_3}{}\,p(u'_{\nu}(t))(v_{\nu} - u'_{\nu}(t))\,d \Gamma + \int\limits_{\Gamma_3}\,\Big( \int\limits_{0}^{t}\,b(t-s) (u^{'}_{\nu})^{+}(s)\,ds \Big)(v_{\nu} - u'_{\nu}(t))\,d \Gamma\\
&+ \int\limits_{\Gamma_3}^{}\,j^0_{\tau}(t,u'_{\tau}(t);v_{\tau} - u'_{\tau}(t))\,d \Gamma \geqslant \langle f(t), v - u'(t)\rangle_{V^* \times V}
\end{split}
\end{align*}
for all $v \in K$ and a.e. $t \in (0,T)$ with $u(0) = u_0,$ where $\sigma^I(u) \in L^2(0,T;\h)$ is the unique function defined in Lemma \ref{OPHISIGMA}.
\end{problem}
{\bf{Step 2.}} Let $u' = w$. We define the operator $\r: \v \longrightarrow \v$ such that
\begin{equation}\label{OPHIR}
(\r w)(t) = \int\limits_{0}^{t}\,w(s)\,ds + u_0\ {\rm for}\ w \in \v,\ \mbox{a.e.}\ t \in (0,T).
\end{equation}
Hence, Problem \ref{model2I} can be formulated as follows.

\begin{problem}\label{ap1}
Find $w \in \v$ such that $w(t) \in K,\ \sigma \in L^2(0,T;\h)$ and
\begin{align}\label{AP1}
\begin{split}
\sigma(t) = \a(t,\varepsilon(w(t))) + \b(t,\varepsilon((\r w)(t)) + \sigma^I((\r w)(t))\ \mbox{a.e.}\ t \in (0,T)\qquad \qquad \quad
\end{split}
\end{align}
\begin{align}\label{AP2}
\begin{split}
&\big(\sigma(t) , \varepsilon(v) - \varepsilon(w(t)) \big)_{\h} 
+ \int\limits_{\Gamma_3}{}\,p(w_{\nu}(t))(v_{\nu} - w_{\nu}(t))\,d \Gamma\\ &+ \int\limits_{\Gamma_3}\,\Big( \int\limits_{0}^{t}\,b(t-s) w^{+}_{\nu}(s)\,ds \Big)(v_{\nu} - w_{\nu}(t))\,d \Gamma + \int\limits_{\Gamma_3}^{}\,j^0_{\tau}(t,w_{\tau}(t);v_{\tau} - w_{\tau}(t))\,d \Gamma\\ &\geqslant \langle f(t), v - w(t)\rangle_{V^* \times V}\ \ \mbox{for all}\ v \in K\ \mbox{and a.e.}\ t \in (0,T).
\end{split}
\end{align}
\end{problem}

Combining \eqref{AP1} and \eqref{AP2}, we obtain the following problem.

\begin{problem}\label{ap}
Find $w \in \v$ such that $w(t) \in K$ and
\begin{align*}
\begin{split}
&\big( \a(t,\varepsilon(w(t))), \varepsilon(v) - \varepsilon(w(t)) \big)_{\h} +\big( \b(t,\varepsilon((\r w)(t))) + \sigma^I((\r w)(t)), \varepsilon(v) - \varepsilon(w(t)) \big)_{\h}\\
&+ \int\limits_{\Gamma_3}{}\,p(w_{\nu}(t))(v_{\nu} - w_{\nu}(t))\,d \Gamma + \int\limits_{\Gamma_3}\,\Big( \int\limits_{0}^{t}\,b(t-s) w^{+}_{\nu}(s)\,ds \Big)(v_{\nu} - w_{\nu}(t))\,d \Gamma\\
&+ \int\limits_{\Gamma_3}^{}\,j^0_{\tau}(t,w_{\tau}(t);v_{\tau} - w_{\tau}(t))\,d \Gamma \geqslant \langle f(t), v - w(t)\rangle_{V^* \times V}
\end{split}
\end{align*}
for all $v \in K$ and a.e. $t \in (0,T).$
\end{problem}
Next, we introduce the operator $A \colon (0,T) \times V \longrightarrow V^*$ defined by
\begin{equation}\label{opB}
\langle A(t,u),v\rangle_{V^* \times V} = \big( \a (t, \varepsilon(u)),\varepsilon(v) \big)_{\h}
\end{equation}
for all $u,v \in V$ and a.e. $t\in (0,T).$ The operator $A$ satisfies \eqref{wA}(b)--(d) with $m_A = m_{\a} >$ $0$, $a_0(t)= \sqrt{2}\, \Vert \overline{a}_0(t) \Vert_{L^2(\Omega)},\ a_1 = \sqrt{2}\, \overline{a}_1 > 0$ (see \cite{MOSBOOK}, p.~205 and \cite{MOS13}, p.~3394). Now, we prove the property \eqref{wA}(f). Let $u_0 \in K$ be given. Using the Cauchy-Schwartz inequality and the conditions \eqref{wkA}(d),(e), we get
\begin{align*}
\begin{split}
\langle A(t,u), u &- u_0 \rangle_{V^* \times V} = \big(\a(t,\varepsilon(u)),\varepsilon(u) - \varepsilon(u_0)\big)_{\h} = \big(\a(t,\varepsilon(u)),\varepsilon(u)\big)_{\h} + \big(\a(t,\varepsilon(u)),- \varepsilon(u_0)\big)_{\h}\\
&\geqslant \alpha_{\a}\,\Vert \varepsilon(u) \Vert_{\h}^2 - \Vert  \a(t,\varepsilon(u))\Vert_{\h}\,\Vert \varepsilon(u_0) \Vert_{\h} = \alpha_{\a}\,\Vert u \Vert_V^2 - \overline{a}_1\,\Vert u \Vert_V - \overline{a}_0(t)\,\Vert u_0 \Vert_V.
\end{split}
\end{align*}
So, we conclude that \eqref{wA}(f) holds with $\alpha_A = \alpha_{\a},\ \beta = \overline{a}_1$ and $\beta_1(t) = \overline{a}_0(t)\,\Vert u_0 \Vert_V$.

We also define the operators $\s_1, \s_2,\s_3 \colon \v \longrightarrow \v^*$  by
\begin{align}\label{ops1}
\begin{split}
&\langle (\s_1 w)(t),v \rangle_{V^* \times V} = \big( \b(t,\varepsilon((\r w)(t))), \varepsilon(v) \big)_{\h}
\\
&\langle (\s_2 w)(t),v \rangle_{V^* \times V} =  \big( \sigma^I((\r w)(t)), \varepsilon(v) \big)_{\h},
\\
&\langle (\s_3 w)(t),v \rangle_{V^* \times V} = \int\limits_{\Gamma_3}^{}\,\big( \int\limits_{0}^{t}\,b(t-s)w^+_\nu(s)\,ds \big)v_\nu\,d\Gamma
\end{split}
\begin{minipage}{0.1\linewidth}
$\left.
\begin{tabular}{c}
 \\ \\ \\ \\
\end{tabular}
\right\}$
\end{minipage}
\end{align}
for all $w \in \v,\ v \in V$ and a.e. $t\in (0,T).$ The hypotheses \eqref{wB}, \eqref{wC}, \eqref{wopb} and the definition \eqref{ops1} imply that the following inequalities hold (cf. \cite{CMOS}).

\begin{align*}
\big( \b (t, \varepsilon ((\mathcal{R} w_1)(t) )) - \b (t, \varepsilon ((\mathcal{R} w_2)(t) )), \varepsilon(v) \big)_{\h} \leqslant L_{\b}\,\Big(\int\limits_{0}^{t}\,\Vert w_1(s) - w_2(s) \Vert_V\,ds \Big)\Vert v \Vert_V,
\end{align*}
\begin{align*}
\begin{split}
&\big( \sigma^I((\r w_1)(t)) - \sigma^I((\r w_2)(t)), \varepsilon(v) \big)_{\h} \leqslant c\,T\,\Big(  \int_{0}^{t}\,\Vert w_1(s) - w_2(s) \Vert_V\,ds \Big)\,\Vert v \Vert_V,
\end{split}
\end{align*}
\begin{align*}
\begin{split}
\int\limits_{\Gamma_3}^{} \Big(\int\limits_{0}^{t}& b(t-s) (w_{1 \nu}^+ (s) - w_{2 \nu}^+(s))\, ds \Big)v_\nu\, d \Gamma\\ &\leqslant \Vert b \Vert_{L^1(0,T;L^\infty(\Gamma_3))}\,\Vert \gamma \Vert^2 \Big(\int\limits_{0}^{t}\,\Vert w_1(s) - w_2(s) \Vert_V\,ds \Big)\Vert v \Vert_V
\end{split}
\end{align*}
for $w_1$, $w_2 \in \v$, $v \in V$, a.e. $t \in (0,T)$.
Hence, the operators $\s_1, \s_2$ and $\s_3,$ defined by \eqref{ops1} satisfy \eqref{wS} with $L_{\s_1} = L_{\b}, L_{\s_2} = c\,T$ and $L_{\s_3} = \Vert \gamma \Vert^2 \Vert b \Vert_{L^1(0,T; L^{\infty}(\Gamma_3))},$ respectively. Moreover, from Lemma \ref{PwS}, we conclude that the operator $\s \colon \v \longrightarrow \v^*$ defined by
$
\langle (\s w)(t),v \rangle_{V^* \times V} = \sum_{i=1}^3\,\langle (\s_i w)(t), v\rangle_{V^* \times V}
$
for all $w \in \v,\ v \in V$ and a.e. $t\in (0,T)$ satisfies \eqref{wS} with 
$L_{\s} = L_{\b} + c\,T + \Vert \gamma \Vert^2 \Vert b \Vert_{L^1(0,T; L^{\infty}(\Gamma_3))}.$
Next, we define the operator $P \colon V \longrightarrow V^*$ by
\begin{equation}\label{opp}
\langle P(u), v \rangle_{V^* \times V} = \int\limits_{\Gamma_3}{}\,p(u_{\nu})v_{\nu}\,d \Gamma
\end{equation}
for all $u,v \in V.$  From \eqref{wpa}(b) and the H\"older inequality, we see that
\begin{align*}
\begin{split}
&\langle P(u) - P(v), u - v \rangle_{V^* \times V} \leqslant \int\limits_{\Gamma_3}^{}\,(p(u_\nu) - p(v_\nu))(v_\nu - u_\nu)\,d\Gamma\\ &\leqslant \Vert p(u_\nu) - p(v_\nu)\Vert_{L^2(\Gamma_3)}\,\Vert u_\nu - v_\nu\Vert_{L^2(\Gamma_3)} \leqslant L_p\,\Vert u_\nu - v_\nu\Vert_{L^2(\Gamma_3)}\,\Vert u_\nu - v_\nu\Vert_{L^2(\Gamma_3)}\\
&\leqslant L_p\,\Vert \gamma \Vert^2\,\Vert u - v \Vert_{V}\,\Vert u - v \Vert_{V}.
\end{split}
\end{align*}
Hence, we conclude that the operator $P$ is Lipschitz continuous with $L_P = L_p \, \Vert \gamma \Vert^2$.

Finally, we define the functional $J \colon (0,T) \times L^2(\Gamma_3; \R^d) \longrightarrow \R$ by
\begin{equation}\label{opJ}
J(t,u)= \int\limits_{\Gamma_3}^{}\,j_\tau(x,t, u_\tau(x))\,d \Gamma
\end{equation}
for all $u \in L^2(\Gamma_3; \R^d)$ and a.e.\,$t \in (0,T).$ Under the assumption \eqref{wj} the functional $J\colon (0,T) \times L^2(\Gamma_3; \R^d) \longrightarrow \R$ defined above satisfies \eqref{wJ} with
$
c_0 = \sqrt{2\, {\rm meas}(\Gamma_3)}\,\overline{c}_0,\ c_1 = \sqrt{2}\,\overline{c}_1, \ d_0= \overline{d}_0 \geqslant 0\ \mbox{and}\ m_J = \alpha_j \Vert \gamma \Vert^2.
$ (see \cite{MOS1}, p. 280).
Under the above notation Problem \ref{model2} can be written in the following equivalent form.
\begin{equation}\label{ll}
\left.
\begin{array}{l}
\hspace{-0.2cm} 
 {\rm Find}\ w \in \v\ {\rm such\ that}\ w(t) \in K\ {\rm and}\\ [2mm]
\langle A(t,w(t)), v - w(t) \rangle_{V^* \times V} + \langle P(w(t)), v - w(t) \rangle_{V^* \times V} \\ [2mm]
+\langle (\s w)(t), v - w(t) \rangle_{V^* \times V} + J^0(t,\gamma w(t);\gamma v - \gamma w(t)) \geqslant \langle f(t), v - w(t)\rangle_{V^* \times V}\\[2mm]
{\rm for\ all}\ v \in K\ {\rm and\ a.e.}\ t \in (0,T). 
\end{array}
\right\}
\end{equation}
We introduce the function $\tilde \varphi \colon V^* \times K \times K \longrightarrow \R$ defined by
\begin{equation}\label{operatorp}
\tilde \varphi (z, u, v) = \langle z, v \rangle_{V^* \times V} + \langle u, v \rangle_{V^* \times V}
\end{equation}
for all $z \in V^*, u,v \in K.$
Hence and from the Cauchy-Schwartz inequality, we have
\begin{align*}
\begin{split}
\tilde \varphi(z_1,&u_1,v_2)-\tilde \varphi(z_1,u_1,v_1)+\tilde\varphi(z_2, u_2,v_1)-\tilde \varphi(z_2,u_2,v_2)\\ &= \langle z_1 - z_2, v_2 - v_1 \rangle_{V^* \times V} + \langle u_1 -  u_2, v_2 - v_1 \rangle_{V^* \times V}\\ &\leqslant  (\Vert z_1 - z_2 \Vert_{V^*} + \Vert u_1 - u_2 \Vert_V )\Vert v_2 - v_1 \Vert_V
\end{split} 
\end{align*}
for all $z_1, z_2 \in V^*, u_1,u_2,v_1,v_2 \in K$. Thus, the condition \eqref{fi01} holds with $\alpha_{\tilde \varphi} = 1$. 
Using the definition of the function \eqref{operatorp} and the fact that $M = \gamma$, Problem \ref{ll} has the following form.
\begin{equation}\label{nier}
\left.
\begin{array}{l}
\hspace{-0.2cm} 
{\rm Find}\ w \in \v\ {\rm such\ that}\ w(t) \in K\ {\rm and} \\ [2mm]
 \langle A(t,w(t)), v - w(t) \rangle_{V^* \times V} + \tilde \varphi((\s w)(t),w(t),v)\\[2mm]
 -\tilde \varphi((\s w)(t),w(t),w(t))+ J^0(t,M w(t);M v - M w(t)) \geqslant \langle f(t), v - w(t)\rangle_{V^* \times V}\\[2mm]
 {\rm for\ all}\ v \in K\ {\rm and\ a.e.}\ t \in (0,T). 
\end{array}
\right\}
\end{equation}
We observe that the condition \eqref{nier4} implies \eqref{WF}(b) with $m_A = m_{\a}$, $\alpha_A = \alpha_{\a},$ $\alpha_{\varphi} = \mbox{max}\,\lbrace 1, L_P \rbrace$, $m_J = \alpha_j$ and $M = \gamma$. Now, applying  Theorem \ref{twierdzenie11} (cf. Section \ref{s31}), we deduce  that there exists a unique function $w \in \v$  that solves \eqref{nier}. From this and the definitions \eqref{opB}, \eqref{ops1}, \eqref{opp}, \eqref{opJ} and \eqref{operatorp}, we deduce that the pair $(w,\sigma) \in \v \times L^2(0,T;\h)$ is a solution to Problem \ref{ap1}. Let $u(t) = (\mathcal{R} w)(t)$ for a.e.\;$t \in (0,T)$ and $w=u'$. Thus, we conclude that the pair $(u, \sigma) \in \w \times L^2(0,T;\h)$ solves Problem \ref{model2I}.  Hence and Lemma \ref{OPHISIGMA}, we deduce that the pair $(u, \sigma) \in \w \times L^2(0,T;\h)$ is a solution to Problem \ref{model2}. The proof of the theorem is complete.
\end{proof}

A couple of functions $(u,\sigma)$ which satisfies \eqref{constitutivelaw}--\eqref{po} is called a {\it weak solution} to Problem \ref{model2}. We conclude that, under
the assumptions of Theorem \ref{twierdzeniemodel1}, Problem \ref{model2} has a unique weak solution with regularity $u \in W^{1,2}(0,T;V)$ and $\sigma \in L^2(0,T;\h)$. We observe, that the regularity of the stress field is, in fact, $\sigma \in L^2(0,T;\h_1)$. Indeed, using \eqref{equationofmotion} and \eqref{wf1f2}, we deduce that Div$\sigma \in L^2(0,T;L^2(\Omega;\R^d))$ and hence $\sigma \in L^2(0,T;\h_1).$

Research supported by the Marie Curie International Research Staff Exchange Scheme Fellowship within the 7th European Community Framework Programme under Grant Agreement No. 295118 and the National Science Center of Poland under the Maestro Project no. DEC-2012/06/A/ST1/00262.

\end{document}